\documentclass[11pt,a4paper]{article}
\usepackage[T1]{fontenc}
\usepackage{lmodern,amsmath,amssymb,amsthm,mathtools,microtype,enumitem,xcolor}
\usepackage[margin=28mm]{geometry}
\usepackage{fancyhdr}
\usepackage[colorlinks=true,linkcolor=blue!45!black,citecolor=blue!45!black,urlcolor=blue!45!black]{hyperref}
\hypersetup{pdftitle={The silp--spli equality for group rings},pdfauthor={}}
\numberwithin{equation}{section}
\newtheorem{theorem}{Theorem}[section]
\newtheorem{proposition}[theorem]{Proposition}
\newtheorem{lemma}[theorem]{Lemma}
\newtheorem{corollary}[theorem]{Corollary}
\theoremstyle{definition}\newtheorem{example}[theorem]{Example}
\theoremstyle{remark}\newtheorem{remark}[theorem]{Remark}
\DeclareMathOperator{\silp}{silp}\DeclareMathOperator{\spli}{spli}\DeclareMathOperator{\sfli}{sfli}
\DeclareMathOperator{\FPD}{FPD}\DeclareMathOperator{\FID}{FID}
\DeclareMathOperator{\pd}{pd}\DeclareMathOperator{\id}{id}\DeclareMathOperator{\fd}{fd}
\DeclareMathOperator{\Gpd}{Gpd}\DeclareMathOperator{\Gid}{Gid}\DeclareMathOperator{\Gfd}{Gfd}
\DeclareMathOperator{\Ggldim}{Ggldim}\DeclareMathOperator{\Gwgldim}{Gwgldim}
\DeclareMathOperator{\Ghd}{Ghd}\DeclareMathOperator{\Gcd}{Gcd}
\DeclareMathOperator{\Ext}{Ext}\DeclareMathOperator{\Tor}{Tor}\DeclareMathOperator{\Hom}{Hom}
\newcommand{\Z}{\mathbb Z}\newcommand{\Q}{\mathbb Q}
\newcommand{\op}{\mathrm{op}}\newcommand{\D}{\mathrm D}\newcommand{\E}{\mathcal E}
\newcommand{\Lot}{\otimes^{\mathbf L}}\newcommand{\RHom}{\mathbf R\!\Hom}
\setlist[enumerate]{label=\textup{(\roman*)},leftmargin=2em}
\title{The silp--spli equality for group rings}\author{Xiaolei Zhang \thanks{School of Mathematics and Statistics, Tianshui Normal University, Tianshui 741000, China; E-mail: zxlrghj@163.com}\and Guocheng Dai \thanks{Corresponding author} \thanks{School of Mathematical Sciences, Sichuan Normal University, Chengdu 610066, China; E-mail: ggcdai@163.com}}\date{}
\begin{document}\maketitle
\begin{abstract}
Let $R$ be an arbitrary associative ring with an identity. In this paper, we show that the equality $\silp(R)=\spli(R)$ implies $\silp(RG)=\spli(RG)$ for any group $G$.
\end{abstract}

\section{Introduction}
Let $A$ be an associative ring with an identity. In their work on complete cohomology, Gedrich and Gruenberg \cite{GG} introduced
\begin{align*}
\silp(A)&=\sup\{\id_A P:P\text{ is a projective left }A\text{-module}\},\\
\spli(A)&=\sup\{\pd_A I:I\text{ is an injective left }A\text{-module}\}.
\end{align*}
The connection of these two invariants with the famous Gorenstein symmetry conjecture is very close. For an Artin algebra $\Lambda$, the usual duality on finitely generated modules together with the structure of arbitrary injective modules gives
\begin{equation}\label{artin-identities}
	\silp(\Lambda)=\id_\Lambda\Lambda,
	\qquad
	\spli(\Lambda)=\id_{\Lambda^{\op}}\Lambda;
\end{equation}
see \cite{BR,Emm2012}. The Gorenstein symmetry conjecture asserts
\begin{equation}\label{gsc}
	\id_\Lambda\Lambda<\infty
	\quad\Longleftrightarrow\quad
	\id_{\Lambda^{\op}}\Lambda<\infty;
\end{equation}
see \cite[Conjecture 13]{ARS} and \cite[Note after VII.2.7]{BR}. When both self-injective dimensions are finite they agree, so \eqref{gsc} for $\Lambda$ is equivalent to $\silp(A)=\spli(A)$ for the Artin $\Lambda$, infinity allowed. Furthermore,
 whether
\begin{equation}\label{symmetry}
\silp(A)=\spli(A)
\end{equation}
holds for every ring $A$ is also still open. When both invariants are finite they coincide, so the real question is whether exactly one of them can be finite. The present paper studies how \eqref{symmetry} behaves when $R$ is replaced by a group ring $RG$.

We begin by recalling the comparisons that bear on this question. Jensen proved \eqref{symmetry} for commutative Noetherian rings \cite[Corollary 5.9]{Jensen}. Gedrich and Gruenberg proved $\silp(kG)\leq\spli(kG)$ for $k$ commutative Noetherian of finite self-injective dimension \cite[Theorem 2.4]{GG}. The relation between the two invariants over group rings was later taken up by Dembegioti and Talelli \cite{DT} and by Emmanouil \cite{Emm2010}; in particular, Emmanouil established the equality for integral group rings. Emmanouil and Talelli connected the flat dimensions of injectives with $\silp$, $\spli$ and the finitistic dimensions \cite{ET}; their Theorem 3.7 gives the two-sided comparison for $\aleph_0$-Noetherian rings.

The Gorenstein global dimension gives a convenient way to phrase these comparisons. Combining results of Beligiannis and Reiten \cite{BR}, Bennis and Mahdou \cite{BM}, Emmanouil \cite{Emm2012} and Bouchiba \cite{Bouchiba}, one has
\begin{equation}\label{intro-gglobal}
\Ggldim(A)=\max\{\silp(A),\spli(A)\}.
\end{equation}
In particular, finite $\silp$ and finite $\spli$ share the value $\Ggldim(A)$. Christensen, Estrada, Liang, Thompson and Wang refined this by showing
\begin{align}
\Ggldim(A)&=\max\{\spli(A),\FPD(A)\}\notag\\
&=\max\{\silp(A),\FID(A)\}\notag\\
&=\max\{\sfli(A),\FPD(A)\};\label{intro-finitistic}
\end{align}
see \cite[Theorem 1.18]{CELTW}. Here $\FPD$ and $\FID$ are the big finitistic projective and injective dimensions, and $\sfli$ is the supremum of the flat dimensions of injective left modules. The  symmetry of Gorenstein weak global dimension is due to Christensen, Estrada and Thompson \cite{CET}.

Dalezios and Emmanouil obtained further comparisons using projectively coresolved Gorenstein flat modules \cite{DE}. Their Corollary 22 says that
\begin{equation}\label{intro-DE}
\spli(A)<\infty,\quad\sfli(A^{\op})<\infty
\quad\Longrightarrow\quad
\silp(A)=\spli(A)<\infty.
\end{equation}
Proposition 23 there applies when $\spli$ is finite on both sides. Corollary 24 gives $\silp(A)\leq\spli(A)$ for $A\cong A^{\op}$, with equality when $\spli(A)$ is finite. They also prove equality for rings isomorphic to their opposites whose one-sided ideals are countably generated \cite[Corollaries 27 and 28]{DE}. In particular,
\begin{equation}\label{intro-commgroup}
\silp(kG)=\spli(kG)
\end{equation}
for every group $G$ and every commutative $\aleph_0$-Noetherian ring $k$ \cite[Remark 25(ii)]{DE}. Neither countability of $G$ nor finiteness of its cohomological dimension is needed here.

There are also comparison theorems based on coherence. Wang,  Li,  Hu and  Zhu proved equality for a two-sided coherent ring isomorphic to its opposite \cite[Corollary 2.10(i)]{WLHZ}. In particular, every commutative coherent ring satisfies \eqref{symmetry}. Wang and Yang proved a two-sided maximum formula for generalized coherent rings and deduced equality for commutative generalized coherent rings \cite[Theorem 3.11 and Corollary 3.16]{WY}. A ring is left generalized coherent if the flat dimensions of its level right modules admit a finite uniform bound, and left weak coherence means that products of flat right modules have finite flat dimension. Chatzistavridis and Emmanouil proved
\begin{equation}\label{intro-sum}
\silp(A)+\silp(A^{\op})=\spli(A)+\spli(A^{\op})
\end{equation}
for two-sided weakly coherent rings \cite[Theorem 3.4]{CE}, and Dai and Zhang extended this to two-sided $\aleph_1$-coherent rings \cite[Theorem 3.3]{DZ}; their Corollary 3.4 gives \eqref{symmetry} for such rings when $A\cong A^{\op}$. Chatzistavridis and Emmanouil also show that $\silp(A)=\spli(A)<\infty$ forces $\silp(A^{\op})=\spli(A^{\op})$, where the latter common value may be infinite \cite[Theorem 3.2]{CE}. All of these statements concern extended integers. None of them asserts that finiteness of a single left invariant alone gives equality over an arbitrary non-commutative coherent ring.

A different approach uses acyclic complexes. Wang, Li, Hu and Zhu give conditions on acyclic complexes of projectives, injectives and flats which yield comparisons among $\silp$, $\spli$ and the two values of $\sfli$ \cite[Theorem 1.2]{WLHZ}. Their Corollary 1.3 gives equality under a condition on character modules of injective right modules together with the stated acyclicity assumptions. The coherent and generalized coherent cases above are covered by these results. We will use these older comparison theorems, as well as the homological facts collected in Section~2, without proof.

Our theorem keeps the equality as a hypothesis on the coefficient ring.
\begin{theorem}\label{intro-main}
Let $R$ be an associative ring and let $G$ be any group. If $\silp(R)=\spli(R)$, then
\[
\silp(RG)=\spli(RG).
\]
\end{theorem}

All rings and modules in this paper are unital. Modules are left modules unless a side is specified, and right modules are modules over the opposite ring.  In derived categories we use cohomological indexing: a module $M[-1]$ sits in degree $1$ and $M[1]$ in degree $-1$.

\section{Preliminaries and change of rings}

A Gorenstein projective module is a cycle in an exact complex of projectives which stays exact after applying $\Hom_A(-,P)$ for every projective $P$; Gorenstein injective modules are defined dually. A Gorenstein flat left module is a cycle in an exact complex of flat left modules which stays exact after tensoring with every injective right module. The corresponding dimensions are written $\Gpd$, $\Gid$ and $\Gfd$. We use
\[
\Ggldim(A)=\sup_M\Gpd_A M,\qquad
\Gwgldim(A)=\sup_M\Gfd_A M.
\]
The big finitistic dimensions are
\[
\FPD(A)=\sup\{\pd_A M:\pd_A M<\infty\},\qquad
\FID(A)=\sup\{\id_A M:\id_A M<\infty\}.
\]
Finally, $\sfli(A)=\sup_I\fd_A I$, the supremum being taken over injective left modules.

\begin{proposition}\label{facts}
For every ring $A$:
\begin{enumerate}
\item $\FPD(A)\leq\silp(A)$ and $\FID(A)\leq\spli(A)$.
\item
\begin{align}
\Ggldim(A)&=\max\{\silp(A),\spli(A)\}\label{maxclass}\\
&=\max\{\sfli(A),\FPD(A)\}.\label{maxflat}
\end{align}
If $\silp(A)$ and $\spli(A)$ are finite, they are equal.
\item $\Gwgldim(A)\leq\Ggldim(A)$. Gorenstein weak global dimension is symmetric; when finite it equals both $\sfli(A)$ and $\sfli(A^{\op})$.
\item If $\pd_A M<\infty$, then $\Gid_A M=\id_A M$. If $\id_A M<\infty$, then $\Gpd_A M=\pd_A M$.
\item If $\FPD(A)<\infty$, every flat left $A$-module has finite projective dimension.
\item If $\spli(A)<\infty$ and $\sfli(A^{\op})<\infty$, then $\silp(A)=\spli(A)<\infty$.
\end{enumerate}
\end{proposition}
For (i), (ii) and (iv), see \cite[Proposition 1.3]{Emm2012}, \cite{Holm,HolmGI} and \cite[Theorems 1.18 and 2.2]{CELTW}. Part (iii) follows from \cite{CET} and \cite[Theorem 5.13]{CELTW}; see also \cite[Section 1.2]{DE}. For (v) see \cite{Jensen,Jorgensen}, and for (vi) see \cite[Corollary 22]{DE}. We will apply these results on either side of the ring by replacing $A$ with $A^{\op}$.

Three more consequences will be needed below. If $\silp(A)=n<\infty$, then
\begin{equation}\label{finitepdid}
\pd_A M<\infty\quad\Longrightarrow\quad\id_A M\leq n;
\end{equation}
see \cite[Proposition 1.3]{Emm2012}. This applies in particular to flat modules when $\FPD(A)<\infty$, by Jensen's theorem. Next,
\begin{equation}\label{twosfli}
\sfli(A),\sfli(A^{\op})<\infty
\quad\Longrightarrow\quad\silp(A)=\spli(A),
\end{equation}
with an infinite common value allowed; this is just Proposition~\ref{facts}(i), (ii) and (vi) combined. Finally, for an injective module $I$ of finite Gorenstein flat dimension,
\begin{equation}\label{injgf}
\fd_A I=\Gfd_A I.
\end{equation}
This follows from the finite-dimensional approximation theorem of \cite{EmmApprox}. We also use closure under extensions and direct summands for the class of modules of Gorenstein flat dimension at most a fixed integer; the relevant resolving and approximation properties over arbitrary rings are in \cite{SS,EmmApprox}.

\begin{lemma}[Subgroup monotonicity]\label{subgroup}
For $H\leq G$, one has $\silp(RH)\leq\silp(RG)$ and $\spli(RH)\leq\spli(RG)$.
\end{lemma}
This is the standard restriction, induction and coinduction comparison; see \cite{Emm2010,GG}. The point is that $RG$ is free over $RH$ on both sides and the coefficient projection $RG\to RH$ is a bimodule map.

\begin{lemma}\label{localization}
If $T=S^{-1}A$ for a multiplicative subset of $Z(A)$, then
\[
\silp(T)\leq\silp(A),\qquad\spli(T)\leq\spli(A).
\]
For $T$-modules $X,Y$ there are also natural isomorphisms
\begin{equation}\label{extlocal}
\Ext_T^j(X,Y)\cong\Ext_A^j(X,Y).
\end{equation}
\end{lemma}
\begin{proof}
The functors $T\otimes_A-$ and $-\otimes_A T$ are exact and $T\otimes_A T\cong T$. If $P_\bullet\to X$ is an $A$-projective resolution of a $T$-module $X$, then $T\otimes_A P_\bullet\to X$ is a $T$-projective resolution. By adjunction the cochain complexes $\Hom_T(T\otimes_A P_\bullet,Y)$ and $\Hom_A(P_\bullet,Y)$ are isomorphic, and their cohomology gives \eqref{extlocal}.

An injective $T$-module $I$ is injective over $A$: for an exact sequence of $A$-modules, apply the exact functor $T\otimes_A-$ and then $\Hom_T(-,I)$. If $\spli(A)=b<\infty$, take an $A$-projective resolution of $I$ of length at most $b$; localizing it gives a $T$-projective resolution of $I$, so $\pd_T I\leq b$. This proves the $\spli$ comparison.

Suppose instead that $\silp(A)=a<\infty$. Proposition~\ref{facts}(i) gives $\FPD(A)\leq a$. A $T$-projective module $P$ is $A$-flat, being a summand of a direct sum of copies of the $A$-flat module $T$, so Jensen's theorem, Proposition~\ref{facts}(v), gives $\pd_A P<\infty$, and \eqref{finitepdid} gives $\id_A P\leq a$. For every $T$-module $X$, \eqref{extlocal} then gives $\Ext_T^j(X,P)=0$ for $j>a$, so $\id_T P\leq a$. This proves the $\silp$ comparison. If the right-hand invariant is infinite, the inequality is automatic.
\end{proof}

We use the standard derived-category constructions and adjunctions of \cite[Chapter 10]{Weibel}. For a bounded complex $X$ supported in degrees $u,\ldots,v$ we will use the following estimates:
\begin{align}
\id_A X^i\leq a\text{ for all }i
&\ \Longrightarrow\
\Hom_{\D(A)}(M,X[j])=0\quad(j>a+v),\label{complex-id}\\
\pd_A X^i\leq b\text{ for all }i
&\ \Longrightarrow\
\Hom_{\D(A)}(X,M[j])=0\quad(j>b-u).\label{complex-pd}
\end{align}
They follow from the finite filtration by the shifts $X^i[-i]$, and we apply them to any object represented by such a complex. A retraction in $\D(A)$ preserves these vanishing statements, since applying either Hom functor preserves a retraction.

\section{Weak characteristic modules}
For a commutative ring $k$, put $\Ghd_kG=\Gfd_{kG}k$ with $k$ acting trivially. A weak characteristic module is a $k$-flat $kG$-module $W$ of finite $kG$-flat dimension together with a $k$-pure $kG$-monomorphism $k\to W$.

If $\sfli(k)<\infty$ and $h=\Ghd_kG<\infty$, such a $W$ exists with $\fd_{kG}W=h$ \cite[Corollary 3.8]{KS}. For every commutative $k$,
\begin{equation}\label{ghdlower}
\Ghd_kG\leq\sfli(kG)
\end{equation}
by \cite[Proposition 6.1]{KS}. We will also use
\begin{equation}\label{comm-equality}
\silp(kG)=\spli(kG)
\end{equation}
when $k$ is commutative and $\aleph_0$-Noetherian \cite[Remark 25(ii)]{DE}.

\begin{lemma}\label{tensorbound}
Let $R$ be an associative $k$-algebra with $k$ acting centrally. If $W$ is a weak characteristic module of $kG$-flat dimension $h$, then
\[
\sfli(RG)\leq\sfli(R)+h,\qquad
\sfli((RG)^{\op})\leq\sfli(R^{\op})+h.
\]
\end{lemma}
\begin{proof}
We prove the left inequality; if $\sfli(R)=\infty$ there is nothing to show. Put $s=\sfli(R)<\infty$ and $A=RG$. Let $E$ be an injective left $A$-module. Since $A_R$ is free, the restriction of $E$ to $R$ is injective, so $\fd_R E\leq s$.

For a left $kG$-module $L$, give $L\otimes_k E$ the $A$-action
\[
(rg)(\ell\otimes e)=g\ell\otimes rge.
\]
This is well defined because $k$ is central and group elements commute with $R$. The map
\begin{equation}\label{untwist}
kG\otimes_k E\longrightarrow A\otimes_R E,
\qquad g\otimes e\longmapsto g\otimes g^{-1}e,
\end{equation}
is an $A$-isomorphism; its inverse sends $rg\otimes e$ to $g\otimes rge$, and the two composites are easily checked on the displayed elements. A flat $R$-resolution of $E$ of length $s$ becomes, after applying the exact functor $A\otimes_R-$, a flat $A$-resolution of the target of \eqref{untwist}. Hence that target has flat dimension at most $s$.

By the filtered-colimit description of flat modules \cite{EJ}, a flat $kG$-module $L$ is a filtered colimit of finitely generated free $kG$-modules. Tensoring over $k$ commutes with this colimit, and a filtered colimit of modules of flat dimension at most $s$ again has flat dimension at most $s$: tensor a projective resolution of a right $A$-module with the colimit, compute $\Tor$, and use exactness of filtered colimits. So
\begin{equation}\label{flat-tensor-estimate}
\fd_A(L\otimes_k E)\leq s
\quad\text{whenever }L\text{ is flat over }kG.
\end{equation}

Take a flat resolution
\[
0\longrightarrow L_h\longrightarrow L_{h-1}\longrightarrow\cdots
\longrightarrow L_0\longrightarrow W\longrightarrow0
\]
over $kG$. Each $L_i$ is $k$-flat, and $W$ is $k$-flat by hypothesis. Starting from the epimorphism onto $W$ and reading off the Tor exact sequence, the kernels in this resolution are successively $k$-flat, so each short exact sequence is $k$-pure. Tensoring with $E$ therefore gives an exact sequence of $A$-modules. Applying \eqref{flat-tensor-estimate} and dimension shifting along this finite sequence yields
\[
\fd_A(W\otimes_k E)\leq s+h.
\]
The pure monomorphism $k\to W$ gives an $A$-monomorphism $E\to W\otimes_k E$, which splits since $E$ is injective. Flat dimension does not increase on a direct summand, so $\fd_A E\leq s+h$. As $E$ was arbitrary, this proves the left inequality.

The right inequality follows by applying the same argument to $R^{\op}$ and using the ring isomorphism $(RG)^{\op}\cong R^{\op}G$, $g\mapsto g^{-1}$. The same weak characteristic module can be used after this identification.
\end{proof}

\begin{proposition}\label{coefficient-test}
Let $k$ be commutative and $\aleph_0$-Noetherian with $\sfli(k)<\infty$, and let $R$ be a central $k$-algebra. Suppose that for every $G$ there are finite constants $u,v$ with
\[
\silp(kG)\leq\silp(RG)+u,\qquad
\spli(kG)\leq\spli(RG)+v.
\]
Then $\silp(R)=\spli(R)$ implies $\silp(RG)=\spli(RG)$ for every $G$.
\end{proposition}
\begin{proof}
If the common value for $R$ is infinite, Lemma~\ref{subgroup} with $H=1$ makes both invariants of $RG$ infinite. So assume $\silp(R)=\spli(R)=d<\infty$. By Proposition~\ref{facts}(ii) and (iii), $\sfli(R),\sfli(R^{\op})\leq d$. Fix $G$ and set $h=\Ghd_kG$.

If $h<\infty$, the weak characteristic-module theorem \cite[Corollary 3.8]{KS} supplies the module needed in Lemma~\ref{tensorbound}, so both $\sfli(RG)$ and $\sfli((RG)^{\op})$ are at most $d+h$, and \eqref{twosfli} gives the required equality, whether its common value is finite or infinite.

If $h=\infty$, then \cite[Proposition 6.1]{KS} gives $\sfli(kG)=\infty$. Since flat dimension is at most projective dimension, $\spli(kG)=\infty$. The ring $k$ is commutative and $\aleph_0$-Noetherian, so \cite[Remark 25(ii)]{DE} gives $\silp(kG)=\infty$ as well. The assumed comparison inequalities, with $u$ and $v$ finite, then force $\silp(RG)=\spli(RG)=\infty$. These cases are exhaustive.
\end{proof}

\section{Coefficients of finite characteristic}
\begin{lemma}\label{finitechar-restriction}
Suppose $R$ has characteristic $n>1$ and put $k=\Z/n\Z$, $S=kG$, $A=RG$. Restriction from $A$ to $S$ sends projectives to Gorenstein projectives and injectives to Gorenstein injectives.
\end{lemma}
\begin{proof}
The ring $k=\Z/n\Z$ is quasi-Frobenius. We use the standard complete-resolution properties over a quasi-Frobenius ring \cite{EJ}: every module is Gorenstein projective and Gorenstein flat, projective modules are injective, and injective modules are projective.

Take a totally acyclic complex $P_\bullet$ of projective $k$-modules with $R$ as a cycle. The functor $S\otimes_k-$ is exact because $S$ is $k$-free, its values on the terms are projective $S$-modules, and the corresponding cycle is $S\otimes_k R\cong{}_SA$. For projective $Q$ over $S$, the restriction of $Q$ to $k$ is projective and
\[
\Hom_S(S\otimes_k P_\bullet,Q)
\cong\Hom_k(P_\bullet,Q)
\]
is exact. Thus ${}_SA$ is Gorenstein projective. A projective $A$-module is a summand of a direct sum of copies of $A$, and Gorenstein projectives are closed under sums and summands \cite{Holm}, so restriction sends every projective $A$-module to a Gorenstein projective $S$-module.

For the injective assertion, work first on the right. Induce a complete flat resolution of $R_k$ along $k\to S$; its terms are flat right $S$-modules and its cycle is $A_S$. If $L$ is an injective left $S$-module, its restriction to $k$ is injective because induction $S\otimes_k-$ is exact, and is therefore projective over $k$. Tensoring the induced resolution with $L$ is thus exact, which shows that $A_S$ is Gorenstein flat.

The character module of a Gorenstein flat right module is Gorenstein injective on the left \cite{Holm,EJ}, so ${}_SA^+$ is Gorenstein injective, where $(-)^+=\Hom_\Z(-,\Q/\Z)$. The left $A$-module $A^+$ is an injective cogenerator, so every injective left $A$-module embeds in a product of copies of $A^+$ and splits off by injectivity. After restriction it is still a direct summand of that product, and products and summands of Gorenstein injectives are Gorenstein injective \cite{Holm,EJ}. This proves the claim for injectives.
\end{proof}

\begin{proposition}\label{finitechar}
For a ring $R$ of characteristic $n>1$ and every group $G$,
\[
\silp((\Z/n\Z)G)\leq\silp(RG),\qquad
\spli((\Z/n\Z)G)\leq\spli(RG).
\]
Consequently $R$ satisfies the conclusion of Theorem~\ref{intro-main}.
\end{proposition}
\begin{proof}
Let $k=\Z/n\Z$, $S=kG$ and $A=RG$. Since $n$ is the characteristic of $R$, the unit map $k\hookrightarrow R$ is injective, and since $k$ is self-injective as a $k$-module there is a $k$-linear retraction $R\to k$. Extending on coefficients gives an $S$-bimodule retraction $A\to S$. In particular, for every $S$-module $M$, the unit and the coefficient retraction exhibit $M$ as a direct summand of the induction, and evaluation at $1$ followed by precomposition with the retraction does the same for the coinduction.

Suppose first that $a=\silp(A)<\infty$. For a projective $S$-module $P$, the module $A\otimes_S P$ is projective over $A$, so it has an injective coresolution of length at most $a$ over $A$. By Lemma~\ref{finitechar-restriction}, restricting this coresolution gives a coresolution by Gorenstein injective $S$-modules, hence
\[
\Gid_S(A\otimes_S P)\leq a.
\]
Since $P$ is a direct summand, $\Gid_S P\leq a$. Now $P$ has finite projective dimension, so Proposition~\ref{facts}(iv) gives $\id_S P=\Gid_S P\leq a$. Taking the supremum over $P$ gives the first inequality.

Suppose next that $b=\spli(A)<\infty$. For an injective $S$-module $I$, put $E=\Hom_S(A,I)$, which is injective over $A$. Restricting a projective $A$-resolution of $E$ of length at most $b$ and applying Lemma~\ref{finitechar-restriction} gives a Gorenstein projective $S$-resolution of a module of which $I$ is a direct summand, so $\Gpd_S I\leq b$. Since $I$ has finite injective dimension, namely zero, Proposition~\ref{facts}(iv) gives $\pd_S I=\Gpd_S I\leq b$, and the second inequality follows. Each comparison is automatic if the right-hand side is infinite.

The ring $k$ is commutative Noetherian with $\sfli(k)=0$, so Proposition~\ref{coefficient-test} applies with $u=v=0$ and yields the final assertion.
\end{proof}

\section{Pure unit embeddings over a localization of the integers}
\begin{lemma}\label{restriction}
Let $k$ be a localization of $\Z$, $R$ a central $k$-algebra, and $S=kG$, $A=RG$. Then $\pd_S A\leq1$ and $\pd_{S^{\op}}A\leq1$. Projective $A$-modules restrict to $S$-modules of projective dimension at most one, and injective $A$-modules restrict to $S$-modules of injective dimension at most one.
\end{lemma}
\begin{proof}
Since $k$ is a localization of the PID $\Z$, its global dimension is at most one. A $k$-projective resolution $0\to P_1\to P_0\to R\to0$ induces an $S$-projective resolution of ${}_SA\cong S\otimes_k R$, and the right-hand version gives the same bound for $A_S$. A projective $A$-module restricts to a summand of a direct sum of copies of ${}_SA$ and therefore has $S$-projective dimension at most one.

Let $E$ be injective over $A$, let $X$ be an $S$-module and take an $S$-projective resolution $P_\bullet\to X$. Termwise adjunction gives
\[
\Hom_S(P_\bullet,E)\cong\Hom_A(A\otimes_S P_\bullet,E).
\]
Since $E$ is $A$-injective, the cohomology on the right is $\Hom_A(\Tor_j^S(A,X),E)$, which vanishes for $j>1$ because of the right projective bound on $A$. Hence $\id_S E\leq1$.

We will also need the resulting estimates for modules of finite dimension over $A$. If $\id_A M\leq a$, restrict an injective coresolution of length $a$; each term has $S$-injective dimension at most one, so the long exact Ext sequences give $\id_S M\leq a+1$. Similarly, a projective resolution over $A$ gives
\begin{equation}\label{restriction-estimates}
\id_S M\leq\id_A M+1,\qquad
\pd_S M\leq\pd_A M+1
\end{equation}
whenever the right-hand side is finite.
\end{proof}

\begin{lemma}\label{squarezero}
Suppose $k\hookrightarrow R$ is pure, where $k$ is a localization of $\Z$. Let $C=R/k$ and let $f:C[-1]\to k$ represent its extension class in $\D(k)$. Then $f\Lot_k f=0$.
\end{lemma}
\begin{proof}
Let $T$ be the torsion submodule of $C$ and write $F=C/T$. We first show that the restriction of the unit extension to $T$ splits. The torsion submodule $T_R$ of the $k$-module $R$ injects into $C$, because $k$ is torsion-free, and its image lies in $T$.

For the reverse inclusion, let $c=r+k\in T$. There is a nonzero $a\in k$ with $ar\in k$, and purity gives $aR\cap k=ak$. So $ar=ab$ for some $b\in k$, and $r-b$ is torsion and maps to $c$. Hence $T_R\to T$ is bijective. Being $k$-linear, it has a $k$-linear inverse, so we get a section $T\to R$ and the unit extension has zero restriction in $\Ext_k^1(T,k)$.

Apply $\Hom_k(-,k)$ to $0\to T\to C\to F\to0$. Exactness gives
\[
\Ext_k^1(F,k)\longrightarrow\Ext_k^1(C,k)
\longrightarrow\Ext_k^1(T,k),
\]
and the class of the unit extension lies in the kernel of the second map. Choose a preimage represented in $\D(k)$ by $\beta:F[-1]\to k$. If $\pi:C\to F$ is the quotient, then
\[
f=\beta\circ\pi[-1].
\]
Now $F$ is torsion-free over a PID, hence flat, and its derived tensor square coincides with the ordinary one. The tensor square of $\beta$ is a morphism
\[
(F\otimes_k F)[-2]\longrightarrow k,
\]
and its group of possible values is $\Ext_k^2(F\otimes_k F,k)=0$ because $\operatorname{gl.dim}k\leq1$. Therefore
\[
f\Lot_k f=(\beta\Lot_k\beta)\circ
(\pi[-1]\Lot_k\pi[-1])=0.\qedhere
\]
\end{proof}

\begin{proposition}\label{pure}
Let $k$ be a localization of $\Z$ and suppose the central unit embedding $k\hookrightarrow R$ is pure. Then for every group $G$,
\begin{equation}\label{purebounds}
\silp(kG)\leq\silp(RG)+2,\qquad
\spli(kG)\leq\spli(RG)+2.
\end{equation}
Consequently $R$ satisfies the conclusion of Theorem~\ref{intro-main}.
\end{proposition}
\begin{proof}
Write $C=R/k$ and $I=C[-1]$. The short exact unit sequence gives a distinguished triangle
\begin{equation}\label{unit-triangle}
I\xrightarrow{f}k\longrightarrow R\longrightarrow I[1]
\end{equation}
in $\D(k)$. Let $K$ be a cone of $f\Lot_k f:I\Lot_k I\to k$. Lemma~\ref{squarezero} says $f\Lot_k f=0$, so the cone triangle is split: there are maps
\[
k\xrightarrow{\iota}K\xrightarrow{\rho}k,
\qquad\rho\iota=1_k
\]
in $\D(k)$.

Consider the composition
\[
I\Lot_k I\xrightarrow{1_I\otimes f}I\xrightarrow{f}k.
\]
The cone of the first map is $I\Lot_k R$, by tensoring \eqref{unit-triangle} with $I$; the cone of the second is $R$; and the cone of the composition is $K$. So the octahedral axiom yields a distinguished triangle
\begin{equation}\label{octa}
(R\Lot_k C)[-1]\longrightarrow K\longrightarrow R\longrightarrow.
\end{equation}
Here we used symmetry of the derived tensor product over the commutative ring $k$. These constructions do not require $R$ to be $k$-flat.

Put $S=kG$ and $A=RG$. We prove the $\silp$ comparison first. Assume $a=\silp(A)<\infty$ and let $P$ be a projective left $S$-module. Its restriction is $k$-projective, so tensoring \eqref{octa} with $P$, with $G$ acting on $P$, gives a triangle in $\D(S)$:
\begin{equation}\label{tensor-triangle}
D[-1]\longrightarrow K\Lot_k P\longrightarrow R\otimes_k P\longrightarrow,
\qquad D=R\Lot_k C\Lot_k P.
\end{equation}
Choose a projective $k$-resolution $0\to C_1\to C_0\to C\to0$. Since $P$ is $k$-projective, $D[-1]$ is represented by
\[
R\otimes_k C_1\otimes_k P\longrightarrow R\otimes_k C_0\otimes_k P
\]
in degrees $0$ and $1$. Each $C_i\otimes_k P$ is projective over $S$ (a direct summand of a direct sum of copies of $P$), so each displayed term is projective over $A$, using $R\otimes_k(-)\cong A\otimes_S(-)$ on $S$-modules. The module $R\otimes_k P$ is $A$-projective as well.

By the definition of $a$ and \eqref{restriction-estimates}, all of these $A$-projective modules have $S$-injective dimension at most $a+1$. Equation~\eqref{complex-id} then gives, for every $S$-module $Y$,
\begin{align*}
\Hom_{\D(S)}(Y,D[-1][j])&=0&& (j>a+2),\\
\Hom_{\D(S)}(Y,(R\otimes_k P)[j])&=0&& (j>a+1).
\end{align*}
The long exact Hom sequence of \eqref{tensor-triangle} now yields
\[
\Hom_{\D(S)}(Y,(K\Lot_k P)[j])=0\qquad(j>a+2).
\]
Tensoring $\rho\iota=1_k$ with $P$ makes $P$ a retract of $K\Lot_k P$ in $\D(S)$, and the maps are $S$-linear because $G$ acts on the factor $P$. Thus $\Ext_S^j(Y,P)=0$ for $j>a+2$, so $\id_S P\leq a+2$. Taking the supremum over $P$ gives the first inequality.

We turn to the $\spli$ comparison. Assume $b=\spli(A)<\infty$ and let $E$ be an injective left $S$-module. It restricts to an injective $k$-module because $S$ is $k$-free. Apply $\RHom_k(-,E)$ to \eqref{octa}; this yields the triangle
\begin{equation}\label{homtriangle}
U\longrightarrow L\longrightarrow V[1]\longrightarrow
\end{equation}
in $\D(S)$, where
\[
U=\Hom_k(R,E),\quad L=\RHom_k(K,E),\quad
V=\RHom_k(R\Lot_k C,E).
\]
The group acts on each Hom complex through its action on $E$. Applying the contravariant functor to the retraction of $k$ from $K$ makes $E$ a retract of $L$ in $\D(S)$.

Set $X=\Hom_k(C,E)$. We claim $\id_S X\leq1$. If $P_\bullet\to Y$ is an $S$-projective resolution, then
\[
\Hom_S(P_\bullet,X)\cong\Hom_S(C\otimes_k P_\bullet,E).
\]
The restriction of $P_\bullet$ is a $k$-projective resolution, so the $j$th cohomology of the right-hand complex is $\Hom_S(\Tor_j^k(C,Y),E)$, which vanishes for $j>1$ because $\operatorname{gl.dim}k\leq1$. This proves the claim.

Derived tensor--Hom adjunction gives
\begin{equation}\label{derived-adjunction}
V\simeq\RHom_k(R,X)\simeq\RHom_S(A,X)
\end{equation}
as objects of $\D(S)$. To identify the last term with the appropriate action, choose an injective $S$-resolution of $X$; its terms restrict to injective $k$-modules, and termwise coinduction identifies $\Hom_S(A,-)$ with $\Hom_k(R,-)$, so the same resolution computes both sides. Choosing this resolution of length one, coinduction sends its terms to injective $A$-modules and $V$ is represented by a complex of injective $A$-modules in degrees $0$ and $1$. Also $U\cong\Hom_S(A,E)$ is injective over $A$.

Every term just obtained has $A$-projective dimension at most $b$, so by \eqref{restriction-estimates} its $S$-projective dimension is at most $b+1$. The complex $V[1]$ is supported in degrees $-1,0$, so \eqref{complex-pd} gives
\[
\Hom_{\D(S)}(V[1],Y[j])=0\quad(j>b+2),
\qquad
\Ext_S^j(U,Y)=0\quad(j>b+1).
\]
The long exact Hom sequence of \eqref{homtriangle} then implies
\[
\Hom_{\D(S)}(L,Y[j])=0\qquad(j>b+2)
\]
for every $S$-module $Y$. Since $E$ is a retract of $L$, this gives $\Ext_S^j(E,Y)=0$ for $j>b+2$, hence $\pd_S E\leq b+2$. Taking the supremum over injective $E$ gives the second comparison.

Both comparisons are automatic if the right-hand invariant is infinite. Finally, $k$ is Noetherian with $\sfli(k)\leq1$, so Proposition~\ref{coefficient-test} applies with $u=v=2$.
\end{proof}

\section{Gluing along central idempotents}
Write $A\in\E$ when $\silp(A)=\spli(A)$ as extended integers.
\begin{theorem}\label{gluing}
Let $\{e_i:i\in I\}$ be pairwise orthogonal central idempotents in $A$, put $J=\bigoplus_i Ae_i$ and $C=A/J$. If $Ae_i\in\E$ for every $i$ and $C\in\E$, then $A\in\E$.
\end{theorem}
\begin{proof}
For a finite subset $F\subseteq I$ write $e_F=\sum_{i\in F}e_i$ and set
\[
\mathcal S=\{1-e_F:F\subseteq I\text{ finite}\}.
\]
By orthogonality $(1-e_F)(1-e_H)=1-e_{F\cup H}$, so $\mathcal S$ is a multiplicative set of central elements. Inverting $1-e_{\{i\}}$ kills $e_i$, since their product is zero, while every element of $\mathcal S$ is congruent to $1$ modulo $J$. Hence
\begin{equation}\label{quotient-localization}
C=A/J\cong\mathcal S^{-1}A.
\end{equation}
In particular $C$ is flat on both sides. If $J=A$, some finite $e_F$ equals $1$ and $C$ is zero; the same argument and the bounds below then apply with zero modules omitted.

The category of modules over the central factor $Ae_i$ is a direct factor of the category of $A$-modules, and its projectives and injectives are projective and injective over $A$. Applying the projection $M\mapsto e_iM$ to resolutions also shows that a module supported on this factor has the same projective or injective dimension over $Ae_i$ as over $A$. So both invariants of $Ae_i$ are bounded by the corresponding invariants of $A$. We treat the two finite cases separately.

Assume first that $n=\silp(A)<\infty$. The hypothesis $Ae_i\in\E$ and the factor comparison give
\[
\silp(Ae_i)=\spli(Ae_i)\leq n.
\]
Lemma~\ref{localization}, \eqref{quotient-localization} and $C\in\E$ give the same bound for both invariants of $C$, so Proposition~\ref{facts}(ii) yields
\begin{equation}\label{first-component-bounds}
\Ggldim(Ae_i)\leq n,
\qquad\Ggldim(C)\leq n.
\end{equation}

Let $E$ be an injective left $A$-module. Orthogonality identifies $JE$ with the internal direct sum $\bigoplus_i e_iE$: elements of $J$ have only finitely many nonzero components, and multiplication by $e_j$ picks out the $j$th component. Each $e_iE$ is injective over $Ae_i$ and therefore has projective dimension at most $n$ there; since $Ae_i$-projectives are $A$-projective, the direct sum of the corresponding resolutions gives $\pd_A JE\leq n$. By \eqref{finitepdid}, $\id_A JE\leq n$.

Consider
\begin{equation}\label{leftseq}
0\longrightarrow JE\longrightarrow E\longrightarrow Q\longrightarrow0.
\end{equation}
For $j\geq1$ the Ext sequence gives $\Ext_A^j(Y,Q)\cong\Ext_A^{j+1}(Y,JE)$, so $\id_A Q<\infty$; in fact $\id_A Q\leq\max\{0,n-1\}$. The module $Q$ is annihilated by $J$ and is therefore a $C$-module, and the Ext comparison of Lemma~\ref{localization} gives $\Ext_C^j(Y,Q)\cong\Ext_A^j(Y,Q)$ for every $C$-module $Y$. So $\id_C Q<\infty$.

Proposition~\ref{facts}(iv) together with \eqref{first-component-bounds} gives $\pd_C Q=\Gpd_C Q\leq n$. A $C$-projective resolution of length $n$ restricts to an $A$-flat resolution, since $C$ is $A$-flat, so $\fd_A Q\leq n$. We already have $\fd_A JE\leq\pd_A JE\leq n$, and the Tor sequence of \eqref{leftseq} gives $\fd_A E\leq n$. As $E$ was arbitrary, $\sfli(A)\leq n$. Also $\FPD(A)\leq n$ by Proposition~\ref{facts}(i), and \eqref{maxflat} gives $\Ggldim(A)\leq n$. Both classical invariants are finite and hence equal.

Assume next that $m=\spli(A)<\infty$. The factor and localization comparisons, now applied to $\spli$, give
\begin{equation}\label{cbounds}
\Ggldim(Ae_i)\leq m,
\qquad\Ggldim(C)\leq m.
\end{equation}
By Proposition~\ref{facts}(iii), each of these rings has Gorenstein weak global dimension at most $m$ on either side.

We need one restriction property on the right. Let $H$ be Gorenstein flat over $C^{\op}$ and take a complete flat resolution $F_\bullet$ of $H$ by right $C$-modules. It remains exact as a complex of flat right $A$-modules. If $L$ is an injective left $A$-module, then $\pd_A L\leq m$; tensoring a projective $A$-resolution of $L$ with the flat right $A$-module $C$ gives a projective $C$-resolution of $C\otimes_A L$, hence $\pd_C(C\otimes_A L)\leq m$. An exact complex of flat right $C$-modules remains exact after tensoring with a left module of finite flat dimension; this is the dimension-shifting argument used in \cite[Section 5]{DE}. Consequently
\[
F_\bullet\otimes_A L\cong
F_\bullet\otimes_C(C\otimes_A L)
\]
is exact, which proves that $H$ is Gorenstein flat over $A^{\op}$. Restricting finite resolutions now gives
\begin{equation}\label{gf-restrict}
\Gfd_{A^{\op}}Y\leq\Gfd_{C^{\op}}Y
\end{equation}
for any right $C$-module of finite Gorenstein flat dimension.

Let $E$ be an injective right $A$-module. Each $Ee_i$ is injective over $(Ae_i)^{\op}$, and the symmetric weak-dimension bound following \eqref{cbounds} gives $\fd_{(Ae_i)^{\op}}Ee_i\leq m$. Restricting the corresponding flat resolutions and taking the direct sum gives
\[
\fd_{A^{\op}}EJ\leq m.
\]
The quotient $Q=E/EJ$ is a right $C$-module, so \eqref{cbounds} and \eqref{gf-restrict} give
\[
\Gfd_{C^{\op}}Q\leq m,
\qquad\Gfd_{A^{\op}}Q\leq m.
\]
Using extension closure of modules of Gorenstein flat dimension at most $m$ \cite{SS,EmmApprox} in the sequence
\begin{equation}\label{rightseq}
0\longrightarrow EJ\longrightarrow E\longrightarrow Q\longrightarrow0,
\end{equation}
we get $\Gfd_{A^{\op}}E\leq m$. Since $E$ is injective, \eqref{injgf} gives $\fd_{A^{\op}}E\leq m$, so $\sfli(A^{\op})\leq m$.

Now $\spli(A)<\infty$ and $\sfli(A^{\op})<\infty$, so \cite[Corollary 22]{DE}, recorded as Proposition~\ref{facts}(vi), gives $\silp(A)=\spli(A)$. If neither invariant of $A$ is finite, both are infinite. The theorem follows in all cases.
\end{proof}

\begin{remark}
No infinite-product decomposition of $A$ is being asserted. Nor does the proof identify $C\otimes_A E=E/JE$ with $\Hom_A(C,E)$: the latter is $C$-injective when $E$ is injective, whereas the former need not be. The two finiteness cases above use different arguments precisely to avoid these substitutions.
\end{remark}

\section{Arithmetic decomposition and the main theorem}
\begin{lemma}\label{stable}
If $p$ is prime and $p^aR=p^{a+1}R$ for some $a\geq1$, there is a central idempotent $e$ with $p^ae=0$ and $p$ invertible in $(1-e)R$. If $p$ is not invertible in $R$ then $e\ne0$. This central idempotent is unique among those satisfying the two conditions for some power of $p$.
\end{lemma}
\begin{proof}
Choose $b\in R$ with $p^a=p^{a+1}b$ and set $x=pb$. Then $p^a(1-x)=0$. Multiplying by $b^a$ gives $x^a(1-x)=0$, so $x^a=x^{a+1}$; induction gives $x^a=x^{a+t}$ for every $t\geq0$. In particular $u=x^a$ is idempotent.

Since $p$ is central, $u=p^ab^a$ lies in $p^aR$. On the other hand $p^a=p^ax$ gives $p^a=p^ax^a=p^au$, so $uR=p^aR$; likewise $u\in Rp^a$ and $p^a=up^a$, so $Ru=Rp^a=p^aR$. Hence $uR=Ru$ is a two-sided ideal.

For any $r\in R$, the inclusion $ur\in Ru$ gives $ur(1-u)=0$ and the inclusion $ru\in uR$ gives $(1-u)ru=0$, so $ur=uru=ru$. Thus $u$ is central. Since $xu=u$,
\[
p(bu)=u=(bu)p
\]
in the ring $uR$, whose identity is $u$, so $p$ is invertible there. Taking $e=1-u$ gives $p^ae=p^a-p^au=0$. If $e=0$ the preceding computation makes $p$ invertible in $R$, which proves the nonzero claim.

For uniqueness, suppose $e$ and $f$ are central idempotents satisfying the conditions, possibly with different exponents. In the central factor $e(1-f)R$ a power of $p$ is zero, because it is zero on $eR$, while $p$ is invertible, because it is invertible on $(1-f)R$. Its identity is therefore zero, so $e(1-f)=0$; reversing the roles gives $f(1-e)=0$. Hence $e=ef=f$.
\end{proof}

\begin{proposition}\label{arithmetic}
Suppose $R$ has characteristic zero. Set
\[
\Sigma=\{p\text{ prime}:p\notin R^\times,\ p^aR=p^{a+1}R\text{ for some }a\geq1\}.
\]
Let $e_p$ be as in Lemma~\ref{stable} and put $J=\bigoplus_{p\in\Sigma}Re_p$, $R'=R/J$. Then:
\begin{enumerate}
\item the $e_p$ are pairwise orthogonal and $R'\cong U^{-1}R$, where $U$ is generated by $\Sigma$;
\item $R'$ has characteristic zero;
\item if $q^aR'=q^{a+1}R'$ for a prime $q$ and some $a\geq1$, then $q$ is invertible in $R'$;
\item if $V$ is generated by the primes invertible in $R'$, the embedding $\Z[V^{-1}]\hookrightarrow R'$ is pure.
\end{enumerate}
\end{proposition}
\begin{proof}
For each $p\in\Sigma$ choose $a_p\geq1$ with $p^{a_p}e_p=0$. If $p\ne q$, both $p^{a_p}$ and $q^{a_q}$ annihilate $e_pe_q$; these integers are relatively prime, so Bezout gives $e_pe_q=0$. The sum defining $J$ is therefore an internal direct sum.

Let $U$ be the multiplicative set generated by $\Sigma$. Each $e_p$ maps to zero in $U^{-1}R$. Conversely, suppose $ur=0$ with $u\in U$. Only finitely many primes occur in $u$; call this set $F$ and set $e_F=\sum_{p\in F}e_p$. Every $p\in F$ is invertible on $(1-e_F)R$, since this is a central factor of $(1-e_p)R$, so $u$ is invertible there and $ur=0$ implies $(1-e_F)r=0$. Hence $r\in Re_F\subseteq J$, which proves that the kernel of $R\to U^{-1}R$ is $J$.

Each $p\in\Sigma$ is invertible modulo $J$: it is invertible on $(1-e_p)R$ and $e_p$ is zero in $R/J$. The universal property of localization now gives $R/J\cong U^{-1}R$, which is (i).

For (ii), suppose a nonzero integer $n$ maps to zero in $R'$. Then $n1_R\in J$, so $n1_R\in Re_F$ for some finite $F$. A nonzero product $v$ of the powers $p^{a_p}$, $p\in F$, annihilates $e_F$, so $vn1_R=0$, contradicting characteristic zero since $vn\ne0$. Thus $R'$ has characteristic zero.

For (iii), suppose $q^aR'=q^{a+1}R'$ with $q$ not a unit in $R'$. By (i), $q\notin\Sigma$. The equality of ideals gives $q^a\in q^{a+1}U^{-1}R$; clearing a denominator, including one that kills a possible localization kernel, gives
\[
uq^a\in q^{a+1}R
\quad\text{for some }u\in U.
\]
Since $q\notin\Sigma$, the integer $u$ is relatively prime to $q$. Choose $\alpha,\beta$ with $\alpha u+\beta q=1$ and multiply by $q^a$; then $q^a\in q^{a+1}R$, so $q^aR=q^{a+1}R$. Since $q$ is a nonunit in $R$ (it is a nonunit in the quotient $R'$), this forces $q\in\Sigma$, a contradiction.

Finally let $k=\Z[V^{-1}]$. Its map into $R'$ is injective by (ii) and central because it comes from integer scalars. Over a PID an inclusion $k\subseteq R'$ is pure if and only if
\begin{equation}\label{relative-divisibility}
aR'\cap k=ak\qquad(a\in k);
\end{equation}
this is the relative-divisibility criterion, obtained from diagonal reduction of finite matrices over a PID; see \cite{EJ}. After multiplying by the units used as denominators in $k$, it suffices to consider
\begin{equation}\label{integer-equation}
nr=b1_{R'},\qquad n>0,\quad b\in\Z,
\end{equation}
and to prove $b/n\in k$. For $b=0$ this is clear. Otherwise write $n=gu$, $b=gv$ with $g>0$ and $\gcd(u,v)=1$, and choose $\alpha,\beta\in\Z$ with $\alpha u+\beta v=1$. Equation~\eqref{integer-equation} gives
\[
g=\alpha gu+\beta gv=gu(\alpha+\beta r),
\]
so $g\in guR'$.

Let $p$ divide $u$ and write $g=p^tw$ with $t\geq0$ and $p\nmid w$. The membership above gives $g\in p^{t+1}R'$. Choose $\gamma,\delta$ with $\gamma w+\delta p^{t+1}=1$ and multiply by $p^t$; then
\[
p^t=\gamma g+\delta p^{2t+1}\in p^{t+1}R'.
\]
If $t=0$ this says $p$ is a unit. If $t\geq1$ it gives $p^tR'=p^{t+1}R'$, and (iii) again says $p$ is a unit. Every prime divisor of $u$ is therefore inverted in $k$, so $b/n=v/u\in k$. This proves \eqref{relative-divisibility} and hence (iv).
\end{proof}

Now we are ready to state the main result of this paper.
\begin{theorem}\label{main}
For every ring $R$ and every group $G$,
\[
\silp(R)=\spli(R)\quad\Longrightarrow\quad\silp(RG)=\spli(RG).
\]
\end{theorem}
\begin{proof}
Assume $\silp(R)=\spli(R)$. If the common value is infinite, Lemma~\ref{subgroup} applied to the trivial subgroup gives $\silp(RG)=\spli(RG)=\infty$. So suppose the common value is $d<\infty$.

If $R$ has positive characteristic, Proposition~\ref{finitechar} applies directly; the zero ring causes no difficulty, so we may assume characteristic zero. Apply Proposition~\ref{arithmetic}, with its notation $e_p$, $J$, $R'$ and $k=\Z[V^{-1}]$.

Each $Re_p$ is a central direct factor of $R$, and its projective and injective modules are supported on that factor with the same dimensions over $R$. So both invariants of $Re_p$ are at most $d$, and Proposition~\ref{facts}(ii) makes them equal. The characteristic of $Re_p$ is positive, since a power of $p$ annihilates its identity $e_p$, so Proposition~\ref{finitechar} gives
\[
\silp((Re_p)G)=\spli((Re_p)G).
\]

The quotient $R'$ is a central localization of $R$, so Lemma~\ref{localization} gives $\silp(R'),\spli(R')\leq d$; these are finite and hence equal. Proposition~\ref{arithmetic}(iv) supplies a pure central embedding $k\hookrightarrow R'$, so the pure unit comparison, Proposition~\ref{pure}, gives
\[
\silp(R'G)=\spli(R'G).
\]

Put $A=RG$. The elements $e_p$ are still pairwise orthogonal central idempotents in $A$, and coefficientwise projection gives
\[
Ae_p\cong(Re_p)G,
\qquad
A/\bigoplus_{p\in\Sigma}Ae_p\cong R'G.
\]
For the quotient isomorphism, the kernel consists of finite sums whose coefficients lie in $J$, and finitely many coefficients involve only finitely many summands of $J$. Every factor and the quotient satisfy equality, so Theorem~\ref{gluing} gives $\silp(A)=\spli(A)$.
\end{proof}

\begin{remark}
The hypothesis involves left modules only; no equality is assumed for $R^{\op}$. The right flat-dimension estimates follow from finite left Gorenstein global dimension through Proposition~\ref{facts}(iii), rather than from any presumed symmetry of Gorenstein global dimension.
\end{remark}

\section{Finite values and group extensions}
Theorem~\ref{main} asserts an equality of extended integers; by itself it says nothing about finiteness of the common value. We make this distinction explicit before proving the group-extension estimate. As usual, $\Gcd_\Z G=\Gpd_{\Z G}\Z$ and $\Ghd_\Z G=\Gfd_{\Z G}\Z$, with $\Z$ acting trivially.

\begin{corollary}\label{criterion}
Assume $\silp(R)=\spli(R)=d<\infty$ and $h=\Ghd_\Z G<\infty$. For $A=RG$,
\[
\sfli(A),\sfli(A^{\op})\leq d+h,
\]
and
\[
\silp(A)=\spli(A)<\infty\quad\Longleftrightarrow\quad\FPD(A)<\infty.
\]
As extended integers the common value is $\max\{\sfli(A),\FPD(A)\}$.
\end{corollary}
\begin{proof}
Proposition~\ref{facts}(ii) and (iii) gives $\sfli(R),\sfli(R^{\op})\leq d$. Applying the weak characteristic-module theorem over $\Z$ and then Lemma~\ref{tensorbound} yields the two flat-dimension bounds. By Theorem~\ref{main}, the two classical invariants of $A$ have a common value $t$, and \eqref{maxclass} identifies $t$ with $\Ggldim(A)$, infinity included. Equation~\eqref{maxflat} then gives
\[
t=\max\{\sfli(A),\FPD(A)\}.
\]
Since $\sfli(A)\leq d+h<\infty$, this maximum is finite exactly when $\FPD(A)$ is finite. Both the equivalence and the formula follow.
\end{proof}

\begin{theorem}\label{extension}
Let $1\to N\to G\to Q\to1$ be a group extension. If
\[
\silp(RN)=\spli(RN)=c<\infty,\qquad q=\Gcd_\Z Q<\infty,
\]
then
\[
c\leq\silp(RG)=\spli(RG)\leq c+q.
\]
\end{theorem}
\begin{proof}
By \cite[Lemma 2.2 and Theorem 2.4]{Ren} applied over $\Z$, there are a $\Z$-projective $\Z Q$-module $\Lambda$ and a $\Z$-split $\Z Q$-monomorphism $\eta:\Z\to\Lambda$ with $\pd_{\Z Q}\Lambda=q$. A projective abelian group is free, so choose a projective resolution
\begin{equation}\label{lambda-resolution}
0\longrightarrow L_q\longrightarrow\cdots\longrightarrow L_0
\longrightarrow\Lambda\longrightarrow0
\end{equation}
over $\Z Q$. Every $L_i$ is free as an abelian group and all kernels in the sequence are subgroups of free abelian groups, hence free; thus every short exact sequence in \eqref{lambda-resolution} splits over $\Z$. Both tensoring and applying $\Hom_\Z(-,-)$ will therefore preserve exactness with arbitrary coefficients.

Set $B=RN$ and $A=RG$. For an injective left $A$-module $E$, restriction is injective over $B$ because $A_B$ is free, so $\pd_B E\leq c$. Inflate the $Q$-action on each $L_i$ and on $\Lambda$ to $G$ and give $L_i\otimes_\Z E$ the diagonal action, with $R$ acting on $E$. The map
\begin{equation}\label{extension-induction}
\Z Q\otimes_\Z E\longrightarrow A\otimes_B E,
\qquad\bar g\otimes e\longmapsto g\otimes g^{-1}e,
\end{equation}
is an $A$-isomorphism. If the lift $g$ is replaced by $gn$ with $n\in N$, the image becomes $gn\otimes n^{-1}g^{-1}e=g\otimes g^{-1}e$, so the map does not depend on the lift; its inverse sends $rg\otimes e$ to $\bar g\otimes rge$, and the two formulas also show $A$-linearity.

Inducing a $B$-projective resolution of $E$ shows that the target of \eqref{extension-induction} has $A$-projective dimension at most $c$. Each $L_i$ is a summand of a direct sum of copies of $\Z Q$, so $L_i\otimes_\Z E$ is a summand of a direct sum of copies of the source of \eqref{extension-induction} and has projective dimension at most $c$ over $A$. Tensoring \eqref{lambda-resolution} with $E$ gives an exact sequence of length $q$ ending in $\Lambda\otimes_\Z E$; its terms have projective dimension at most $c$, so dimension shifting gives
\[
\pd_A(\Lambda\otimes_\Z E)\leq c+q.
\]
The map $E\to\Lambda\otimes_\Z E$ induced by $\eta$ is a monomorphism, since $\eta$ splits over $\Z$, and it is $A$-linear because $\eta(1)$ is $Q$-invariant. Injectivity of $E$ splits this monomorphism over $A$, so $\pd_A E\leq c+q$. Taking the supremum,
\begin{equation}\label{extension-spli}
\spli(A)\leq c+q.
\end{equation}

For the other invariant, let $P$ be projective over $A$. Its restriction is projective over $B$, so $\id_B P\leq c$. For any $\Z Q$-module $L$ give $\Hom_\Z(L,P)$ the action
\[
(rf)(\lambda)=rf(\lambda),\qquad
(gf)(\lambda)=g f(g^{-1}\lambda).
\]
Applying $\Hom_\Z(-,P)$ to \eqref{lambda-resolution} gives an exact coresolution
\begin{equation}\label{lambda-coresolution}
0\longrightarrow\Hom_\Z(\Lambda,P)\longrightarrow\Hom_\Z(L_0,P)
\longrightarrow\cdots\longrightarrow\Hom_\Z(L_q,P)\longrightarrow0.
\end{equation}
The basic coinduction identification is
\begin{equation}\label{extension-coinduction}
\Hom_\Z(\Z Q,P)\cong\Hom_B(A,P).
\end{equation}
Writing the left-hand side as functions $f:Q\to P$, the map sends $f$ to the $B$-linear function $h$ determined by $h(g)=g f(\bar g^{-1})$; it is $B$-linear because $N$ acts trivially on $Q$. Its inverse is determined by $f(\bar g)=g h(g^{-1})$, which is independent of the lift. The coinduced action on the right is $(ah)(x)=h(xa)$, and substituting into the two formulas verifies that \eqref{extension-coinduction} respects both the group action and the $R$-action.

Coinduction is exact because ${}_BA$ is free, and it preserves injectives because restriction is exact. Applying it to a $B$-injective coresolution of $P$ of length at most $c$ gives
\[
\id_A\Hom_B(A,P)\leq c.
\]
Since $L_i$ is a summand of a direct sum of copies of $\Z Q$, the module $\Hom_\Z(L_i,P)$ is a summand of a product of copies of the module in \eqref{extension-coinduction}. A product of modules of injective dimension at most $c$ again has injective dimension at most $c$: take the product of their length-$c$ injective coresolutions, using exactness of products and injectivity of products of injectives. So every term after the first in \eqref{lambda-coresolution} has injective dimension at most $c$, and dimension shifting gives
\[
\id_A\Hom_\Z(\Lambda,P)\leq c+q.
\]
Evaluation at $\eta(1)$ defines an $A$-map $\Hom_\Z(\Lambda,P)\to P$, which is surjective: a $\Z$-linear retraction $\lambda:\Lambda\to\Z$ of $\eta$ sends any $p\in P$ to $z\mapsto\lambda(z)p$, which takes the value $p$ at $\eta(1)$. Since $P$ is projective, the epimorphism splits over $A$, so $\id_A P\leq c+q$, and taking the supremum gives $\silp(A)\leq c+q$.

Both invariants of $A$ are now finite, so Proposition~\ref{facts}(ii) makes them equal, while Lemma~\ref{subgroup} applied to $N\leq G$ gives the lower bound $c$.
\end{proof}

\begin{corollary}
If $\silp(R)=\spli(R)=d<\infty$ and $\Gcd_\Z G=q<\infty$, then
\[
d\leq\silp(RG)=\spli(RG)\leq d+q.
\]
For finite $G$ the common value is $d$.
\end{corollary}
\begin{proof}
Take $N=1$ in Theorem~\ref{extension}. For finite $G$ one has $\Gcd_\Z G=0$ \cite[Example 2.5]{Ren}, so the lower and upper bounds coincide.
\end{proof}

\begin{example}\label{example}
Let $G=\bigoplus_{j\geq1}C_2$ and $A=\Q G$. Then
\[
\silp(\Q)=\spli(\Q)=0,\qquad\Ghd_\Z G=0,
\]
but $\silp(A)=\spli(A)=1$.

The group is locally finite, so $\Ghd_\Z G=0$ by the locally finite case of \cite{KS}. Every finite subset of $G$ lies in a finite elementary abelian $2$-subgroup $H$, and $\Q H$ is a finite product of copies of $\Q$, so for every $a\in A$ there is $b\in\Q H\subseteq A$ with $a=a^2b$. Thus $A$ is a commutative von Neumann regular ring. The group and the field are countable, so $A$ is countable.

Let $I$ be an ideal of $A$ and choose generators $a_1,a_2,\ldots$. Every finitely generated ideal in a commutative von Neumann regular ring is generated by an idempotent, so write $(a_1,\ldots,a_n)=Ae_n$; choosing the idempotents so that $e_ne_{n+1}=e_n$, which is possible because the ideals are increasing. Set $d_1=e_1$ and $d_n=e_n-e_{n-1}$ for $n\geq2$. Then $d_n^2=d_n$ and $d_nd_j=0$ for $n\ne j$, and since $Ae_n=\bigoplus_{j\leq n}Ad_j$ we get
\[
I=\bigcup_n Ae_n=\bigoplus_{j\geq1}Ad_j.
\]
Each $Ad_j$ is a direct summand of $A$, so every ideal is projective, and the usual hereditary-ring criterion gives $\operatorname{gl.dim}A\leq1$; see \cite[Chapter 4]{Weibel}.

The augmentation module $\Q$ is not projective. Indeed, an $A$-linear section of the augmentation $\varepsilon:A\to\Q$ would send $1$ to an element $x$ with $gx=x$ for every $g\in G$ and $\varepsilon(x)=1$. For nonzero $x$ let $F$ be its finite support and choose $g\notin FF^{-1}$, possible because $G$ is infinite; then $gF\cap F=\varnothing$, contradicting $gx=x$. Hence no such section exists, and since $\operatorname{gl.dim}A\leq1$ we get $\pd_A\Q=1$. Therefore
\[
1\leq\FPD(A)\leq\silp(A)\leq\operatorname{gl.dim}A\leq1.
\]
Equality with $\spli(A)$ follows either from Theorem~\ref{main} or from \eqref{comm-equality}. So $\silp(RG)\leq\silp(R)+\Ghd_\Z G$ fails even for a field of characteristic zero.
\end{example}

\begin{corollary}\label{artin-application}
Let $\Lambda$ be an Artin algebra for which the Gorenstein symmetry conjecture holds. Then $\silp(\Lambda G)=\spli(\Lambda G)$ for every group $G$. If $\Lambda$ is Iwanaga--Gorenstein of self-injective dimension $d$ and $G$ is finite, the common value is $d$.
\end{corollary}
\begin{proof}
The identities \eqref{artin-identities} together with \eqref{gsc} give $\silp(\Lambda)=\spli(\Lambda)$, infinity included, and Theorem~\ref{main} gives the first assertion. In the second, both coefficient invariants equal $d$, and the finite-group case of the preceding corollary gives the stated value.
\end{proof}

\begin{remark}
Theorem~\ref{main} is a preservation result, not a solution of the general silp--spli symmetry problem for arbitrary rings; its hypothesis is carried throughout. Theorem~\ref{extension} and Example~\ref{example} show that preservation of equality, finiteness of the common value, and a numerical additive formula are genuinely separate questions.
\end{remark}

\end{document}